\documentclass{amsart}

\usepackage{amsfonts,amsmath,amssymb,latexsym,mathrsfs,stmaryrd,tikz-cd,hyperref}
\usepackage{enumitem}  % For custom list labels, e.g., (AP1)

\usepackage{tikz}
\usetikzlibrary{decorations.pathreplacing, calc}

\usepackage{mycommand-v2}
\usepackage{yifengcommand}

\usepackage{ytableau} % drawing young diagrams

\usepackage{longtable} % notations table is long

\usepackage{amsthm,aliascnt,cleveref}

\theoremstyle{plain}
\newtheorem{theorem}{Theorem}[section]

\newaliascnt{lemma}{theorem}
\newtheorem{lemma}[lemma]{Lemma}
\aliascntresetthe{lemma}
\crefname{lemma}{lemma}{lemmas}
\Crefname{lemma}{Lemma}{Lemmas}

\newaliascnt{proposition}{theorem}
\newtheorem{proposition}[proposition]{Proposition}
\aliascntresetthe{proposition}
\crefname{proposition}{proposition}{propositions}
\Crefname{proposition}{Proposition}{Propositions}

\newaliascnt{conjecture}{theorem}
\newtheorem{conjecture}[conjecture]{Conjecture}
\aliascntresetthe{conjecture}
\crefname{conjecture}{conjecture}{conjectures}
\crefname{conjecture}{Conjecture}{Conjectures}

\newaliascnt{claim}{theorem}

\aliascntresetthe{claim}
\crefname{claim}{claim}{claims}
\Crefname{claim}{Claim}{Claims}

\newaliascnt{corollary}{theorem}
\newtheorem{corollary}[corollary]{Corollary}
\aliascntresetthe{corollary}
\crefname{corollary}{corollary}{corollaries}
\Crefname{corollary}{Corollary}{Corollaries}

\newtheorem*{main-theorem}{Main Theorem}
\crefname{main-theorem}{Main Theorem}{Main Theorems}
\Crefname{main-theorem}{Main Theorem}{Main Theorems}

\crefname{goal}{Goal}{Goals}
\Crefname{goal}{Goal}{Goals}

\crefname{wish}{Wish}{Wishes}
\Crefname{wish}{Wish}{Wishes}

\crefname{guess}{Guess}{Guesses}
\Crefname{guess}{Guess}{Guesses}

\newaliascnt{definition}{theorem}
\theoremstyle{definition}
\newtheorem{definition}[definition]{Definition}
\aliascntresetthe{definition}
\crefname{definition}{definition}{definitions}
\Crefname{definition}{Definition}{Definitions}

\newtheorem*{notation}{Notation}

\newaliascnt{construction}{theorem}

\aliascntresetthe{construction}
\crefname{construction}{construction}{constructions}
\Crefname{construction}{Construction}{Constructions}

\newaliascnt{example}{theorem}
\newtheorem{example}[example]{Example}
\aliascntresetthe{example}
\crefname{example}{example}{examples}
\Crefname{example}{Example}{Examples}

\newaliascnt{assumption}{theorem}

\aliascntresetthe{assumption}
\crefname{assumption}{assumption}{assumptions}
\Crefname{assumption}{Assumption}{Assumptions}

\newaliascnt{question}{theorem}

\aliascntresetthe{question}
\crefname{question}{question}{questions}
\Crefname{question}{Question}{Questions}

\newaliascnt{fact}{theorem}

\aliascntresetthe{fact}
\crefname{fact}{fact}{facts}
\Crefname{fact}{Fact}{Facts}

\theoremstyle{remark}
\newtheorem{remark}{Remark}[section]
\crefname{remark}{remark}{remarks}
\Crefname{remark}{Remark}{Remarks}

\usepackage[style=alphabetic]{biblatex}
\usepackage{caption}
\usepackage{subcaption} % Provides the subfigure environment

\usepackage{xcolor}

\definecolor{delete}{cmyk}{0, 0.7808, 0.4429, 0.1412}

\newcommand{\Mn}{\mathbb{M}_n} % matrix space
\makeatletter
\providecommand{\leftsquigarrow}{%
  \mathrel{\mathpalette\reflect@squig\relax}%
}
\newcommand{\reflect@squig}[2]{%
  \reflectbox{$\m@th#1\rightsquigarrow$}%
}
\makeatother

\DeclareMathOperator{\Res}{Res} % Weil restriction
\DeclareMathOperator{\Sym}{Sym} % Symmetric product
\DeclareMathOperator{\UConf}{UConf} % unordered configuration space
\DeclareMathOperator{\Rep}{Rep} % the category of representations
\newcommand{\T}{\mathrm{T}} % diagonal matrices in GLn
\newcommand{\et}{\mathrm{et}} % subscript for etale
\renewcommand{\sp}{\mathrm{sp}} % Principal specialization

\newcommand{\CH}{\mathbf{ch}} % Frobenius Character
\renewcommand{\ch}{\CH} % alias Frobenius Character
\newcommand{\B}{\mathcal{B}} % Flag variety
\renewcommand{\P}{\mathcal{P}} % Partial flag variety
\newcommand{\Gr}{\mathrm{Gr}} % Grassmannian

\newcommand{\uQ}{\underline{\Q}} % the constant sheaf with rational coefficients
\renewcommand{\et}{\mathrm{\acute{e}t}}
\newcommand{\D}{\mathbb{D}} % Verdier duality
\newcommand{\U}{\mathrm{U}} % Unitary group
\newcommand{\UT}{\mathrm{UT}} % real maximal torus

\title{Matrix Points on Varieties}
\author{Asvin G}
\email{gasvinseeker94@gmail.com}

\author{Yifeng Huang}
\email{yifeng.huang@usc.edu}

\author{Ruofan Jiang}
\email{ruofanjiang@berkeley}

\author{Yifan Wei}
\email{yifan.wei@wisc.edu}


\begin{document}

\begin{abstract}
  We study the cohomology of $C_n(X)$, the moduli space of commuting $n$-by-$n$ matrices satisfying the equations defining a quasi-projective scheme $X$. This space can be viewed as a non-commutative Weil restriction from the algebra of $n$-by-$n$ matrices to the ground field. We introduce a semi-simple counterpart $S_n(X)$, defined as the quotient of $X^n \times \GL_n/\T_n$ by the diagonal $S_n$ action. We show that there exists a natural map $\sigma \colon S_n(X) \to C_n(X)$ inducing isomorphism on $\ell$-adic cohomology under mild restrictions on $X$ or the characteristic of the field. This confirms a heuristic derived from Weil restrictions. Furthermore, we provide explicit combinatorial formulae for the Betti numbers of $C_n(X)$ and prove a Macdonald-type generating series. A version for Hermitian matrix point is also proved in the last section.
  % \YW{Rewrite.}
\end{abstract}

\maketitle

% \tableofcontents

%!TEX root=./short_version.tex
\section{Introduction}
\label{sec: the introduction}

\subsection{Motivation: Non-commutative Weil restriction}
Let $X$ be a $k$-scheme, if $L$ is a separable extension of degree $n$ with Galois group $G$, and $\hat{L}$ is its Galois closure, the Weil restriction of $X_L$ to $k$ represents the following functor (for $k$-scheme $S$)
\[
  \Res^L_k(X_L)(S) = \Hom_k(S_L, X)
\]
It can be verified that $\Res^L_k(X_L)$ is represented by the twisted product $X^n \times^G \Spec(\hat{L})$ (the diagonal $G$-quotient of the product $X^n \times \Spec(\hat{L})$, where $G$ acts on the $X^n$ factor by permuting $n = \Hom_k(L, \hat{L})$.) 
% \YH{do you mean Weil restriction of $X_L$? And what does convolution mean, is it a sub or a quotient?}

In this paper, we explore an analogue of this construction in a non-commutative setting. We view the inclusion of a field $k$ into the algebra of $n$-by-$n$ matrices, $\Mn(k)$, as a non-commutative analogue of a degree $n$ ``separable field extension'', and we study the Weil restriction functor from the matrix algebra to the ground field.

Let $X = \Spec R$ be a finite type affine $k$-scheme. We define the functor $C_n(X)$ on commutative $k$-algebras $S$ by
\[
C_n(X)(S) = \Hom_{k\text{-}\mathrm{alg}}(R, \Mn(S)).
\]
If $R$ is generated by commuting variables $x_i$ satisfying relations $P_j(x_i) = 0$, then $C_n(X)$ represents the space of commuting $n$-by-$n$ matrices $X_i$ satisfying the same relations $P_j(X_i) = 0$. This space is representable by a finite type affine scheme.

\begin{example}
  Matrix points recover, as special cases, several classical examples of commuting varieties: $C_n(\A^N)$ is the variety of commuting $N$ tuples in $\mathfrak{gl}_n$, $C_n(\Gm^N)$ is the variety of commuting $N$ tuples in $\GL_n$. When $X = \A^1 - \{a_1, a_2, \dots, a_r\}$, the space $C_n(X)$ parametrizes matrices avoiding eigenvalues $a_1, a_2, \dots, a_r$.
\end{example}

The space $C_n(X)$ admits a moduli interpretation: it parametrizes length $n$ coherent sheaves $M$ on $X$ equipped with a global framing, i.e., an isomorphism $\iota \colon \Gamma(X, M) \isoto k^n$. This interpretation generalizes to arbitrary quasi-projective schemes $X$. As a moduli space $C_n(X)$ is a $\GL_n$ torsor of the Artin stack $\Coh_n(X)$, the moduli space of length $n$ coherent sheaves on $X$.

\subsection{A Semi-simple Counterpart of $C_n(X)$}

There exists a natural map $\hat{\sigma} \colon X^n \times \GL_n/\T_n \to C_n(X)$ which furnishes an ordered direct sum decomposition $\bigoplus_{i = 1}^n L_i = k^n$ with the structure of a semi-simple framed sheaf by equipping $L_i$ with the structure of the sky-scrapper sheaf $k(x_i)$ at $x_i$. For example, when $X = \A^d$, the map $\hat{\sigma}$ takes a $d$-tuple of diagonal $n$-by-$n$ matrices $(D_1, D_2, \dots, D_d)$ and a right coset $[g] \in \GL_n/\T_n$ to the $d$-tuple of commuting matrices $(g D_1 g^{-1}, g D_2 g^{-1}, \dots, g D_d g^{-1})$. 

Note that there exists a diagonal $S_n$ action on $X^n \times \GL_n/\T_n$, under which the map $\hat{\sigma}$ is invariant. Let us denote the quotient space $X^n \times^{S_n} \GL_n/\T_n$ by $S_n(X)$, and the induced map by $\sigma \colon S_n(X) \to C_n(X)$. The image of $S_n(k)$ under $\sigma$ is precisely the set of semi-simple framed sheaves, and over the dense open subscheme $\UConf^n X \subset \Sym^n X$, $\sigma$ is an isomorphism. However globally $\sigma$ is neither injective nor surjective on geometric points. 

\subsection{Main Results}

We show in this paper that $\sigma \colon S_n(X) \to C_n(X)$ identifies the cohomologies of both spaces.

\begin{theorem}
\label{thm: main result}
  Suppose we are in either of the following situations:
  \begin{itemize}
    \item $k$ is of characteristic zero, or
    \item $X$ is locally homogeneous, i.e., any local ring $R$ at a closed point of $X$ admits a non-negative grading.
    % \YH{More explanation needed: is there a grading a priori, and the condition is that the grading is nonnegative? Or, do you mean it is OK as long as $R$ admits a structure of $\mathbb{N}$-graded $k$-algebra?}
  \end{itemize}
  Then the map $\sigma \colon S_n(X) \to C_n(X)$ induces isomorphisms
  \[H_{\et}^i(C_n(X), \Q_\ell) \isoto H_{\et}^i(S_n(X), \Q_\ell)\]
  for all $i$.
\end{theorem}

We conjecture that the assumptions on $X$ or $k$ can be removed (see \Cref{conjecture artin local ring}). For the rest of this paper, we always assume we are in the situation stated in \Cref{thm: main result} unless specified otherwise. Note that the assumption on $X$ automatically holds when $X$ is smooth.
% \YH{The sentence is to give a quick hint that the assumption is mild.}

This result validates the Weil restriction heuristics: we may think of $k \inclu \Mn(k)$ as a non-commutative version of a separable degree $n$ ``field extension''. Let's pretend that $S_n$ is the Galois group of this field extension, and that there exists a Galois closure $\hat{\Mn}(k)$ of the extension $k \inclu \Mn(k)$. Then from the classical theory of Weil restrictions, we expect
\[
  C_n(X) = X^n \times^{S_n} ``\Spec \hat{\Mn}(k)".
\]
On the other hand, the cohomology groups $H^*(\GL_n/\T_n)$ form the regular $S_n$ representation, except that it's graded in some non-trivial way (see \Cref{thm: graded character of the flag variety}). Thus, \Cref{thm: main result} implies that cohomologically, $C_n(X)$ behaves as $X^n \times^{S_n} \GL_n/\T_n$, aligning with the expectation that $H^*(\GL_n/\T_n)$ serves as the cohomological model for the ``Galois closure" of $``\Spec \Mn(k)" \to \Spec k$.

The construction of $S_n(X)$ allows for explicit computation of its cohomology using combinatorial methods involving symmetric functions. We define the Poincare polynomial $P_u(M) = \sum_i \dim H^i(M) (-u)^i$. Let $\phi_n(q) = \prod_{i = 1}^n (1 - q^i)$.

\begin{theorem}
\label{thm: combinatorial formula}
The Poincare polynomial of $C_n(X)$ is given by
\[
P_u(C_n(X)) = \phi_n(u^2) \sp_{u^2}(\ch_u(X^n)),
\]
where $\ch_u(X^n)$ is the graded $S_n$-character on the cohomology of $X^n$ (Definition \ref{def: graded character}), and $\sp_{u^2}$ is the principal specialization of symmetric polynomials $f(x_1, x_2, x_3, \dots) \mapsto f(1, u^2, u^4, \dots)$.
\end{theorem}

Furthermore, we can package these polynomials into a generating series which admits a beautiful infinite product formula in terms of the Betti zeta function $\zeta_X^\mathrm{B}(t) = \sum_{n=0}^\infty P_u(\Sym^n X)t^n$, which itself admits a product factorization by a celebrated formula of Macdonald \cite{macdonald1962poincare}, see \Cref{thm: Betti analog weil conjecture Macdonald formula}.

\begin{theorem}
\label{thm: generating series Betti numbers}
% We have the following product formula (note that $P_u(\Coh_n(X)) = \frac{P_u(C_n(X))}{\phi_n(u^2)}$ by \Cref{formula of cohomology of Cohn}):
The generating series of the Poincar\'e series of $\Coh_n(X)$ is related to the Betti zeta function by an infinite product identity:
  \[
    1 + \sum_{n = 1}^\infty P_u(\Coh_n(X)) t^n = \prod_{i = 0}^\infty \zeta_X^\mathrm{B}(u^{2i}t).
  \]
\end{theorem}
% \YH{Should $i$ start from $1$? Check point counts for smooth curve case.}

% \YH{Interesting point: this implies that for singular curve like $y^2=x^3$, the matrix points have easy cohomology. But the rank $r$ Quot schemes (or, the locus of matrix points parametrizing $R$-modules that can be generated by at most $r$ generators) are now known to have complicated cohomology (due to affine paving in my work with Ruofan). So cohomologically, $\Coh_n(X)$ is not the $r\to \infty$ limit of $\mathrm{Quot}_n(\O_X^r)$. }

\subsection{A Hermitian Version}
Finally, we establish a variant of our main result in the Hermitian setting when $X$ is defined over $\R$ (\Cref{sec: a hermitian variant}). Let $HC_n(X)$ be the space of commuting Hermitian matrices satisfying the equations of $X$, and $HS_n(X) = X(\R)^n \times^{S_n} \U_n/\UT_n$.

\begin{theorem}
\label{thm: hermitian version}
The natural map $\sigma_H \colon HS_n(X) \to HC_n(X)$ induces an isomorphism on rational cohomology.
\end{theorem}

\begin{example}
  We have $\mathrm{HC}_n(\A^N)$ being the real variety of commuting $N$ tuples of Hermitian matrices, and, less obviously but not hard to check, that $\mathrm{HC}_n((S^1)^N)$ is the real variety of commuting $N$ tuples in $U_n$, where $S^1=\Spec \R[x,y]/(x^2+y^2-1)$. Our results thus give the Poincar\'e polynomials of these real varieties, in line with the results of \cite{ramras2019hilbert} in type $A$.
\end{example}

\subsection{Context and Related Works}
The map $\sigma \colon S_n(X) \to C_n(X)$ is $\GL_n$-equivariant, hence it induces a map between $\GL_n$-equivariant cohomologies. Applying our main result to the spectral sequence computing the equivariant cohomology, see that the equivariant cohomologies of $S_n(X)$ and $C_n(X)$ are also identified.
  \[
  \begin{tikzcd}
    E_2^{pq} \colon H^q(B\GL_n, H^p(C_n(X))) \ar[r, Rightarrow] \ar[d, "\sigma^*", "\iso"'] & H^{p + q}_{\GL_n}(C_n(X)) \ar[d, "\sigma^*", "\iso"'] \\
    E_2^{pq} \colon H^q(B\GL_n, H^p(S_n(X))) \ar[r, Rightarrow] & H^{p + q}_{\GL_n}(S_n(X)).
  \end{tikzcd}
  \]
Recall that the equivariant cohomology of a space computes the cohomology of the quotient stack, our result says the rational cohomologies of the stack $\Coh_n(X) = [C_n(X)/\GL_n]$ and the stack $\Sym^n(X \times B\Gm) = [S_n(X)/\GL_n]$ are identified. This is classically known when $X$ is a smooth curve (see, e.g., Heinloth \cite{heinloth2012cohomology}).

% \YW{Just need to talk about equivariant formality, let's not say anything about the derived category on stacks.}
Our result also implies equivariant formality for the $\GL_n$ action on $C_n(X)$ (\Cref{equivariant formality}), that is, the spectral sequence for the equivariant cohomology collapses at the second page. Consequently we obtain the formula for the Poincar\'e polynomial of the stack $\Coh_n(X)$:
\begin{equation}
\label{formula of cohomology of Cohn}
  P_u(\Coh_n(X)) = P_u(B\GL_n) P_u(C_n(X)) = \sp_{u^2}(\ch_u(X^n)).
\end{equation}

% Matrix points recover, as special cases, several classical examples of commuting varieties: $C_n(\A^N)$ is the variety of commuting $N$ tuples in $\mathfrak{gl}_n$, $C_n(\Gm^N)$ is the variety of commuting $N$ tuples in $\GL_n$, $\mathrm{HC}_n(\A^N)$ is the real variety of commuting $N$ tuples in $\mathfrak{u}_n$, and, less obviously but not hard to check, that $\mathrm{HC}_n((S^1)^N)$ is the real variety of commuting $N$ tuples in $U_n$, where $S^1=\Spec \R[x,y]/(x^2+y^2-1)$. Our results thus give the Poincar\'e polynomials of these real varieties, in line with the results of \cite{ramras2019hilbert} in type $A$.

Matrix points on more general varieties have been explored by earlier works of Chien-Hao Liu and Shing-Tung Yau in the context of D-branes in string theory \cite{liu2010d}. The moduli space $C_n(X)$ (as representation space $\mathsf{rep}_n(R)$) is also seen in the works of Lieven Le Bruyn and Jens Hemelaer \cite{le2012representation,hemelaer2016azumaya}. However, to the authors' knowledge, the rational cohomology of the space $C_n(X)$ has not been computed before. 
In fact, even the simple problem of determining the Betti numbers of the space of $n$-by-$n$ matrices avoiding several eigenvalues was already difficult to solve using classical methods, but our combinatorial formula (\Cref{thm: combinatorial formula}) applied to the space $\A^1 - \{a_1, a_2, \dots, a_r\}$ provides a low complexity way to compute these Betti numbers.
% \begin{example}
%   Let $X = \A^1 - \{a_1, a_2, \dots, a_r\}$. The space $C_n(X)$ parametrizes matrices avoiding eigenvalues $a_1, a_2, \dots, a_r$. The Betti numbers of the $C_n(X)$ seem hard to compute using classical methods. 
% \end{example}

\subsection{Acknowledgements}
We are grateful to Dima Arinkin, Jordan Ellenberg, Laurentiu Maxim, Joshua Mundinger, and Yiyu Wang for helpful discussions. We are especially indebted to Tasuki Kinjo for providing the proof idea of base change for $C_n(X) \to \Sym^n X$.

We acknowledge the role of Google's Gemini AI in making us aware of the GIT result by Neeman \cite{neeman1985topology} used in \Cref{contractibility of nil}, and general polish of language in this paper.

Asvin was supported by a grant from the Simons Foundation International [MPS-SIM-00001691, JT].

\subsection{Outline}
In \Cref{sec: the proof}, we analyze the geometry of the map $\sigma$ relatively over $\Sym^n X$ and prove \Cref{thm: main result}. \Cref{sec: combinatorics} develops the necessary combinatorial tools involving graded characters and symmetric functions, leading to the proofs of \Cref{thm: combinatorial formula} and \Cref{thm: generating series Betti numbers}. \Cref{sec: a hermitian variant} discusses the Hermitian variant and proves \Cref{thm: hermitian version}.

%!TEX root=./short_version.tex
\section{The proof}
\label{sec: the proof}
In this section, we prove our main theorem by analyzing the map $\sigma \colon S_n(X) \to C_n(X)$ relatively over the symmetric product $\Sym^n X$. We assume that we are in the situation where either $k$ is of characteristic zero, or that any local ring $R$ at a closed point of $X$ is non-negatively graded. 

\subsection{The Relative Picture}
Let $p_1 \colon S_n(X) \to \Sym^n X$ and $p_2 \colon C_n(X) \to \Sym^n X$ denote the respective natural morphisms. The map $\sigma$ commutes with these projections. We aim to show that $\sigma$ induces an isomorphism on the derived direct images of the constant sheaf.

\begin{theorem}
\label{thm: fiberwise isomorphism}
  Under the assumptions of Theorem \ref{thm: main result}, the map $\sigma$ induces an isomorphism $\sigma^* \colon Rp_{2*} \uQ_\ell \isoto Rp_{1*} \uQ_\ell$ in the derived category $D_c^b(\Sym^n X)$.
\end{theorem}

The heart of our proof is showing that the morphisms $p_1$ and $p_2$ has base change property with respect to the constant sheaf, which allows us to compare stalks easily. 
Theorem \ref{thm: main result} follows immediately from the sheaf theoretic statement by taking cohomology. We spell out the proof for the rest of the section.

\subsection{Cohomology of the Fibers}

Let $z_\lambda$ be a geometric point on $\Sym^n X$ corresponding to a zero-cycle $\lambda_1 [x_1] + \cdots + \lambda_m [x_m]$, where $\lambda = (\lambda_1, \dots, \lambda_m)$ is a partition of $n$. Let $S_\lambda = S_{\lambda_1} \times \cdots \times S_{\lambda_m} \subset S_n$ be the corresponding Young subgroup.

\begin{lemma}
  The underlying variety of the fiber $p_1^{-1}(z_\lambda)$ of $S_n(X) \to \Sym^n X$ is isomorphic to the quotient $(\GL_n/\T_n)/S_\lambda$.
\end{lemma}
\begin{proof}
  We may assume $k$ is algebraically closed. 
  Let $\hat{S}_n(X) = X^n \times \GL_n/\T_n$ where $S_n$ acts diagonally, let $ \pi \colon X^n \to \Sym^n X$ denote the quotient map. Let $\hat{p}_1 \colon \hat{S}_n(X) \to \Sym^n X$ be the composition of the projection $\hat{S}_n(X) \to X^n$ and the quotient $X^n \to \Sym^n X$.
  We have the following commutative diagram:
  \[
  \begin{tikzcd}
    \hat{S}_n(X) \ar[r] \ar[d, "\mu"] & X^n \ar[d, "\pi"] \\
    S_n(X) \ar[r, "p_1"] & \Sym^n X
  \end{tikzcd}
  \]
  Since $\mu$ is a finite \'etale Galois cover, the induced map between fibers $\hat{p}_1^{\,\,-1}(z_\lambda) \to p_1^{-1}(z_\lambda)$ is again a finite \'etale Galois cover. The following diagram is fibered, and vertical arrows are finite \'etale Galois morphisms with Galois group $S_n$,
  \[
  \begin{tikzcd}
    \hat{p}_1^{\,\,-1}(z_\lambda)_{red} \ar[r, hook] \ar[d] & \hat{p}_1^{\,\,-1}(z_\lambda) \ar[d] \\
    p_1^{-1}(z_\lambda)_{red} \ar[r, hook] & p_1^{-1}(z_\lambda)
  \end{tikzcd}
  \]
  Note that $\hat{p}_1^{\,\,-1}(z_\lambda)_{red} = \pi^{-1}(z_\lambda)_{red} \times \GL_n/\T_n$ equipped with a diagonal $S_n$ action. Since $S_n$ acts transitively on $\pi^{-1}(z_\lambda)_{red}$ with $S_\lambda$ being the stabilizer, we have $p_1^{-1}(z_\lambda)_{red} = (\GL_n/\T_n)/S_\lambda$.
\end{proof}

The description of the fiber of $C_n(X) \to \Sym^n X$ is more involved. Let $R_i$ be the local ring of $X$ at $x_i$ with maximal ideal $\mathfrak{m_i}$. Define $\mathrm{Nil}_{\lambda_i, x_i}$ to be the fiber above $\lambda_i[x_i] \in \Sym^{\lambda_i} X$ of the map $C_{\lambda_i}(X) \to \Sym^{\lambda_i} X$. The scheme $\mathrm{Nil}_{\lambda_i, x_i}$ parametrizes length $\lambda_i$ sheaves on $X$ supported at $x_i$ with a framing.

\begin{lemma}
  The fiber $p_2^{-1}(z_\lambda)$ is isomorphic to the associated bundle
  \[\GL_n \times^{\prod_i \GL_{\lambda_i}} \left(\prod_{i = 1}^m \mathrm{Nil}_{\lambda_i, x_i}\right).\]
\end{lemma}
\begin{proof}
  A point in $p_2^{-1}(z_\lambda)$ corresponds to a framed sheaf $(M, \iota)$ whose support is given by $z_\lambda$. Such a sheaf decomposes as $M = \bigoplus M_i$, where $M_i$ is supported at $x_i$ and has length $\lambda_i$. The framing $\iota$ induces a decomposition $k^n = \bigoplus V_i$ where $V_i = \Gamma(M_i)$. The space of such modules $M_i$ with a fixed basis of $V_i$ is parametrized by $\mathrm{Nil}_{\lambda_i, x_i}$. The group $G_\lambda = \prod_i \GL(V_i) \iso \prod_i \GL_{\lambda_i}$ acts on the choice of basis of $V_i$ inside $k^n$. The fiber is thus the quotient of $\GL_n \times \prod_i \mathrm{Nil}_{\lambda_i, x_i}$ by the action of $G_\lambda$, where $G_\lambda$ acts on $\GL_n$ by right multiplication (via the standard inclusion $G_\lambda \subset \GL_n$) and on $\prod \mathrm{Nil}_{\lambda_i, x_i}$ diagonally.
\end{proof}

\begin{conjecture}
\label{conjecture artin local ring}
  If $R$ is an Artinian local algebra over $k$, the underlying variety of $C_n(R)$ has trivial $\ell$-adic cohomology.
\end{conjecture}
This conjecture holds under the assumptions of Theorem \ref{thm: main result}.

\begin{theorem}
\label{contractibility of nil}
  Let $R$ be an Artinian local algebra over $k$. The underlying variety of $C_n(R)$ has trivial $\ell$-adic cohomology if either $k$ has characteristic zero or $R$ is non-negatively graded.
\end{theorem}
\begin{proof}
  If $R$ is non-negatively graded, the grading induces a $\Gm$-action on $R$. This action extends to $C_n(R)$, contracting it to the ``center point'' (the trivial representation where all generators act by zero). Hence, $C_n(R)$ has trivial $\ell$-adic cohomology.

  If $k$ has characteristic zero, we use Geometric Invariant Theory. The $\GL_n$ action on $C_n(R)$ has a unique closed orbit (the trivial representation). By a result of Neeman \cite{neeman1985topology}, there exists an equivariant deformation retraction of the complex points of $C_n(R)$ onto this closed point. Hence $C_n(R)$ has trivial cohomology.
\end{proof}

\begin{corollary}
  Under the assumption of \Cref{thm: main result}, the underlying variety of $\mathrm{Nil}_{\lambda_i, x_i}$ has trivial $\ell$-adic cohomology.
\end{corollary}
\begin{proof}
  Let $R_i$ be the local ring on $X$ at $x_i$, let $\mathfrak{m}_i$ be the maximal ideal of $R_i$. We observe that $C_{\lambda_i}(\Spec R_i/\mathfrak{m}_i^{\lambda_i})$ and $\mathrm{Nil}_{\lambda_i, x_i}$ have the same underlying variety. 
\end{proof}

\begin{proposition}
\label{prop: fiber cohomology}
  The natural map $\sigma \colon S_n(X) \to C_n(X)$ induces isomorphism on rational cohomologies between fibers of $p_1$ and $p_2$.
\end{proposition}
\begin{proof}
Let $z_\lambda$ be a geometric point on $\Sym^n X$. The map $\sigma$ restricted to the fiber $p_1^{-1}(z_\lambda)$ maps surjectively onto this semi-simple locus $Z$. Specifically, $\sigma$ induces a map between $\GL_n$-homogeneous spaces:
\[
  \sigma \colon (\GL_n/\T_n)/S_\lambda \to \GL_n/\prod_i \GL_{\lambda_i}.
\]
This map is a fibration. The typical fiber is isomorphic to
\[
\left(\prod_i \GL_{\lambda_i}\right) / \left(\prod_i N(\T_{\lambda_i})\right) \iso \prod_i (\GL_{\lambda_i}/N(\T_{\lambda_i})).
\]
The spaces $\GL_m/N(\T_m)$ are known to have trivial rational cohomology (see \Cref{thm: graded character of the flag variety}). Therefore, the Leray spectral sequence for the fibration induced by $\sigma$ degenerates, and $\sigma^*$ induces an isomorphism $H^*(p_2^{-1}(z_\lambda)) \isoto H^*(p_1^{-1}(z_\lambda))$. 
\end{proof}

\subsection{Computing the stalks}
Let $z$ be a geometric point on $\Sym^n X$, we show here that the stalks of $Rp_{1*}\uQ_\ell$ and $Rp_{2*}\uQ_\ell$ at $z$ can be identified with the cohomology of the fibers. Note that this is not automatic since neither $p_1$ nor $p_2$ is proper. 

\begin{proposition}[Stalk of $Rp_{1*}\uQ_\ell$]
\label{base change p1}
  Let $z$ be a geometric point on $\Sym^n X$. The base change morphism $(Rp_{1*}\uQ_\ell)_z \to R\Gamma(p_1^{-1}(z), \uQ_\ell)$ is an isomorphism.
\end{proposition}
\begin{proof}
  Let $\hat{S}_n(X)$ denote the product $X^n \times \GL_n/\T_n$, and $\hat{p}_1$ denote the composition of the projection $\hat{S}_n(X) \to X^n$ and the finite quotient $X^n \to \Sym^n X$. We have the following commutative diagram
  \[
  \begin{tikzcd}
    \hat{S}_n(X) \ar[r, "q"] \ar[d, "\mu"] & X^n \ar[d, "\pi"] \\
    S_n(X) \ar[r, "p_1"] & \Sym^n X
  \end{tikzcd}
  \]
  The complex $R\hat{p}_{1*}\uQ_\ell$ carries an $S_n$-action via the $S_n$ quotient map $\mu$, such that $(Rp_{1*}\uQ_\ell)_z = (R\hat{p}_{1*}\uQ_\ell)_z^{S_n}$ for every geometric point $z$ on $\Sym^n X$. It now suffices to show that $(R\hat{p}_{1*}\uQ_\ell)_z = R\Gamma(\hat{p}_1^{\,-1}(z), \uQ_\ell)$ in an $S_n$-equivariant way.

  Applying proper base change to the morphism $\pi$:
  \[
  (R\pi_*Rq_* \uQ_\ell)_z = R\Gamma(\pi^{-1}(z), Rq_*\uQ_\ell)
  \]

  Projection formula implies $Rq_*\uQ_\ell = \uQ_\ell \otimes R\Gamma(\GL_n/\T_n, \uQ_\ell)$, we have:
  \[
    R\Gamma(\pi^{-1}(z), Rq_*\uQ_\ell) = R\Gamma(\pi^{-1}(z), \uQ_\ell) \otimes R\Gamma(\GL_n/\T_n, \uQ_\ell) = R\Gamma(\hat{p}_1^{\, -1}(z), \uQ_\ell)
  \]
  as $\hat{p}_1^{\, -1}(z) = \pi^{-1}(z) \times \GL_n/\T_n$. This concludes that the base change morphism $(R\hat{p}_{1*}\uQ_\ell)_z \to R\Gamma(\hat{p}_1^{\,-1}(z), \uQ_\ell)$ is an isomorphism. Note that $S_n$-equivariance follows from the naturality of the base change morphism.
\end{proof}

Computing the stalk of $Rp_{2*}\uQ_\ell$ is a more difficult problem. The goal is to show that the stalk of $Rp_{2*}\uQ_\ell$ is isomorphic to the cohomology of the fiber of $p_2$. In order to do this, we adopt the approach of \cite{kinjo2024decomposition}, where we approximate the morphism $C_n(X) \to \Sym^n X$ by a morphism related to Quot scheme that has good base change property. The key idea of using Quot scheme is from an email exchange with Tasuki Kinjo, to whom we are deeply grateful.

We introduce a few players in the approximation argument. For each $m \geq n$, let $W_{n, m}$ be the Stiefel variety parameterizing surjective linear maps from a fixed $m$-dimensional vector space to a fixed $n$-dimensional space, it is an open subscheme of the affine space $(\A^n)^m$.
The Stiefel variety $W_{n, m}$ is a principal $\GL_n$-bundle over the Grassmanian $\Gr(m, n)$. The complement of $W_{n, m}$ in $(\A^n)^m$ has codimension $m - n + 1$.

Let $\Gamma_{n, m}$ be the open subscheme of $C_n(X) \times (\A^n)^m$ parametrizing a framed length $n$ sheaf $M$, and $m$ global sections of $M$ which generate $M$ \emph{as an $\sO_X$-sheaf}. The fiber $\Gamma_{n, m}|_v$ of the projection $\Gamma_{n, m} \to C_n(X)$ contains the open subscheme $W_{n, m}$, hence its codimension in $(\A^n)^m$ is bounded below by $m - n + 1$.

Now let $z$ be a geometric point on $\Sym^n X$, consider the following Cartesian diagram:
\[
\begin{tikzcd}
  \Gamma_{n, m}|_z \ar[r, "f_m'"] \ar[d, hook, "t''"] & C_n(X)|_z \ar[r, "p_2'"] \ar[d, hook, "t'"] & \{z\} \ar[d, hook, "t"] \\
  \Gamma_{n, m} \ar[r, "f_m"] & C_n(X) \ar[r, "p_2"] & \Sym^n X
\end{tikzcd}
\]

Now consider the following:
\[
\begin{tikzcd}
  t^*Rp_{2*}\uQ_\ell \ar[r, "d"] \ar[d, "a"]  & R\Gamma(C_n(X)|_z, \uQ_\ell) \ar[d, "c"] \\
  t^*Rp_{2*}Rf_{m*}\uQ_\ell \ar[r, "b"] & R\Gamma(\Gamma_{n, m}|_z, \uQ_\ell)
\end{tikzcd}
\]
We will show that the bottom arrow $b$ is an isomorphism, and the vertical arrows $a, c$ are good approximations, hence $d$ becomes an isomorphism as $m \to \infty$.

First we show $b$ is an isomorphism. The map $p_2 \circ f_m \colon \Gamma_{n, m} \to \Sym^n X$ factorizes as a principal $\GL_n$-torsor $\Gamma_{n, m} \to \mathrm{Quot}_X(\sO^m_X, n)$ followed by a proper morphism $\mathrm{Quot}_X(\sO^m_X, n) \to \Sym^n X$. Consider the diagram:
\[
\begin{tikzcd}
  \Gamma_{n, m}|_z \ar[r, "u'"] \ar[d, hook, "t''"] & \mathrm{Quot}_X(\sO^m_X, n)|_z \ar[r, "v'"] \ar[d, hook, "s"] & \{z\} \ar[d, hook, "t"] \\
  \Gamma_{n, m} \ar[r, "u"] & \mathrm{Quot}_X(\sO^m_X, n) \ar[r, "v"] & \Sym^n X
\end{tikzcd}
\]

Consider the following:
\[
\begin{tikzcd}
  t^*Rv_*Ru_*\uQ_\ell \ar[r, "\iso"] & Rv'_*s^*Ru_*\uQ_\ell \ar[r, "\iso"] & Rv'_*Ru'_*\uQ_\ell = R\Gamma(\Gamma_{n, m}|_z, \uQ_\ell)
\end{tikzcd}
\]
The first isomorphism is proper base change applied to $v$, and the second isomorphism is due to \'etale local triviality of $u$. This establishes $b$ as an isomorphism.

To show $a$ and $c$ are good approximations, we use the following
\begin{lemma}
  Let $\pi \colon E \to X$ be a vector bundle of rank $r$ over a variety $X$, let $j \colon U \subset E$ be an open subscheme of $E$, let $f \colon U \to X$ be the composite $\pi \circ j$. Suppose there exists an integer $N$ such that for all closed point $x \in X$, the codimension of $E|_x \setminus U|_x$ in $E|_x$ is $\geq N$. Then for any complex $\sF \in D_c^{\geq 0}(X)$, we have
  \[
    fib(\sF \to Rf_*f^* \sF) \in D_c^{\geq 2N}(X)
  \]
\end{lemma}
\begin{proof}
  Let $i \colon Z = E\setminus U \inclu E$ be the complement, and $g = \pi \circ i \colon Z \to X$ be the composite map. Note that $\sF \isoto R\pi_*\pi^* \sF$, due to $E$ being a vector bundle. Thus
  \[
  fib(\sF \to Rf_* f^* \sF) = fib(R\pi_* \pi^* \sF \to R\pi_* Rj_*j^*\pi^* \sF) = Rg_*i^!\pi^*\sF
  \]
  Let $C = Rg_*i^!\pi^*\sF$, applying Verdier dual, and the fact that $\pi^! = \pi^*[2r]$, we obtain $\D(C) = Rg_!i^*\pi^*\D(\sF)[2r] = Rg_!g^*\D(\sF)[2r]$. 

  The stalks of $\D(C)$ can now be computed using proper base change applied to $g_!$. Let $\sG = \D(\sF)[2r]$, we have $\sG \in D_c^{\leq - 2r}(X)$. Thus the stalk $\D(C)_x = R\Gamma_c(Z|_x, \sG_x)$ lives in degrees less or equal to $2r - 2N - 2r = -2N$. Therefore $C = fib(\sF \to Rf_* f^* \sF) \in D_c^{\geq 2N}(X)$.
\end{proof}

Applying the above lemma to the map $\Gamma_{n, m} \to C_n(X)$, we see that both $fib(a)$ and $fib(c)$ live in degrees $\geq 2(m - n + 1)$. Combined with the fact that $b$ is an isomorphism, we see that $d$ is an isomorphism in $D_c^b(\{z\})$ as long as $m$ is taken large enough. This concludes the discussion:
\begin{proposition}[Stalk of $Rp_{2*}\uQ_\ell$]
\label{base change p2}
  Let $z$ be a geometric point on $\Sym^n X$. The base change morphism $(Rp_{2*}\uQ_\ell)_z \to R\Gamma(p_2^{-1}(z), \uQ_\ell)$ is an isomorphism.
\end{proposition}

Now combining the two base change results \Cref{base change p1}, \Cref{base change p2} and the computation of fiberwise cohomologies \Cref{prop: fiber cohomology}. We obtain \Cref{thm: fiberwise isomorphism}.

% \YW{Maybe can say a few things about the agreement of fiberwise point counts in the case where $X$ is a smooth curve.}
\begin{example}
  In the case $X = \A^1$, the sheaf isomorphism $\sigma^* \colon Rp_{2*}\uQ_\ell \to Rp_{1*}\uQ_\ell$ has some practical implications. Since both $S_n(\A^1)$ and $C_n(\A^1)$ are smooth affine varieties of dimension $n^2$, one may take a shifted Verdier dual of the isomorphism and obtain an isomorphism between fiberwise compactly supported cohomologies $\sigma_! \colon Rp_{1!\,}\uQ_\ell \to Rp_{2!\,}\uQ_\ell$. In particular, this implies the fibers of $p_1$ and $p_2$ have the same point counts over finite fields. Rational points on the fiber of $p_2 \colon C_n(\A^1) \to \Sym^n \A^1$ correspond to matrices with a given eigenvalue, while rational points on the fiber of $p_1$ correspond to Galois stable decompositions of the $n$-dimensional vector space into direct sums of subspaces. As a special case, we recover the classical fact that the number of $n$-by-$n$ nilpotent matrices over a finite field $\Fq$ equals the number of maximal tori in $\GL_n(\Fq)$, both equals to the number $q^{n^2 - n}$.
\end{example}

\subsection{Equivariant formality}

% \YW{Discuss APness of $C_n(X) \to C_n(Y)$ for proper maps $X \to Y$. In particular, AP implies the IC of $C_n(X)$ is pure for proper $X$.}

% APness is not true, the squaring map GL2(C) -> GL2(C) does not satisfy base change.
As a consequence of the main theorem, we obtain the equivariant formality of $C_n(X)$ under the action of $\GL_n$.

\begin{corollary}
\label{equivariant formality}
The $\GL_n$-equivariant cohomology of $C_n(X)$ is equivariantly formal. In other words the Poincare series of $\Coh_n(X) = [C_n(X)/\GL_n]$ is the product of the Poincare series of $B\GL_n$ and the Poincare series of $C_n(X)$.
\end{corollary}
\begin{proof}
  % It suffices to show $\pi' \colon [S_n(X)/\GL_n] \to B\GL_n$ has strong equivariant formality. \YW{Actually this is another base change we need to show? Need to show this map is etale locally trivial.}

  % Let $\hat{S}_n(X) = X^n \times \GL_n/\T_n$, and $\hat{\pi}' \colon \hat{S}_n(X)/\GL_n \to B\GL_n$. Since $R\pi'_* \uQ_\ell$ is the $S_n$-invariant part of $R\hat{\pi}'_* \uQ_\ell$, it is a direct summand. It suffices to prove that $R\hat{\pi}'_* \uQ_\ell$ splits.

  % Note that $\hat{S}_n(X)/\GL_n = X^n \times B\T_n$. The map $\hat{\pi}'$ factors as $X^n \times B\T_n \to B\T_n \to B\GL_n$. The first map is a trivial fibration. The second map $B\T_n \to B\GL_n$ factors as $B\T_n \to \GL_n\backslash \B_n \to B\GL_n$, where $\B_n$ is the flag variety. The map $B\T_n \to \GL_n\backslash \B_n$ is an affine space bundle (the stacky quotient of $\GL_n/\T_n \to \B_n$), hence acyclic for the constant sheaf. The map $\GL_n\backslash \B_n \to B\GL_n$ is a smooth projective morphism between Artin stacks with affine diagonal. By the decomposition theorem for Artin stacks (see \cite{sun2012decomposition}), the derived direct image splits as a direct sum of shifted cohomology sheaves. Thus $R\hat{\pi}'_* \uQ_\ell$ splits.
  For a connected group $G$, there is a spectral sequence $E_2^{pq}\colon H^p(G, H^q(S, \Q_\ell)) \Rightarrow H^{p + q}_G(S, \Q_\ell)$ computing the $G$-quivariant cohomology of any $G$-scheme $S$. Thus by our main result \Cref{thm: main result}, the $\GL_n$-equivariant cohomologies of $S_n(X)$ and $C_n(X)$ are isomorphic, as their corresponding spectral sequences have isomorphic $E_2$ pages.

  Thus it suffices to show $S_n(X)$ is equivariantly formal. Let us denote $\hat{S}_n(X) = X^n \times \GL_n/\T_n$. Since the $\GL_n$-action on $\hat{S}_n(X)$ commutes with the $S_n$-action, we see that $H^*_{\GL_n}(S_n(X)) = H^*_{\GL_n}(\hat{S}_n(X))^{S_n}$, whence the equivariant formality for $\hat{S}_n(X)$ would imply the equivariant formality for $S_n(X)$. On the other hand, the equivariant cohomology of $\hat{S}_n(X)$ is the tensor product of the usual cohomology of $X^n$ and the equivariant cohomology of $\GL_n/\T_n$, the formality for $\hat{S}_n(X)$ follows from the formality of the equivariant cohomology of $\GL_n/\T_n$.
\end{proof}

%!TEX root=./short_version.tex
\section{Combinatorics}
\label{sec: combinatorics}

This section is dedicated to the character theoretic computation of the cohomologies of the convoluted product $S_n(X) = X^n \times^{S_n} \GL_n/\T_n$. To do so, we develop necessary language to talk about the $S_n$ characters showing up in the cohomology of $H^*(X^n)$ via a power structure on the generating series of graded $S_n$ representations (as $n$ varies). Meanwhile, we observed that the graded $S_n$ characters on the cohomology of $\GL_n/\T_n$ can be packaged nicely using \emph{principal specializations}, this description behaves especially well with power structure, allowing us to derive a compact formula for the cohomology of $S_n(X)$, and an infinite product formula for its generating series as $n$ varies. 

Since the technicalities of \'etale cohomology is irrelevant in this section, we use $\Q$ to denote the char $0$ coefficient field of cohomologies, and use $H^*$ to mean either singular or \'etale cohomology depending on the context.

\begin{notation}
% \YW{Needs fix}
$\text{}$
  \begin{longtable}[h]{ p{0.2\textwidth} p{0.7\textwidth} }
    $\lambda \vdash n$    &     an integer partition $\lambda = (\lambda_1 \geq \lambda_2 \geq \dots \lambda_m)$ of $n$. \\
    $\ell(\lambda)$ or $|\lambda|$      &     the length of a partition \\
    $S_n$    &         the permutation group on $n$ letters \\
    $S_\lambda \subset S_n$    &     Young subgroup $S_{\lambda_1} \times S_{\lambda_2} \times \cdots \times S_{\lambda_m}$ of $S_n$ \\ 
    % $\Res_{S_\lambda}$     &      restriction to the Young subgroup $S_\lambda$ from $S_n$    \\
    % $\Ind_{S_\lambda}$     &      induction from the Young subgroup $S_\lambda$ to $S_n$      \\
    $\mathsf{SymPoly}$     &      the ring of symmetric polynomials with rational coefficients in infinitely many variables $x_1$, $x_2$, $x_3$, $\dots$ \\
    $s_\lambda$          &        Schur basis of $\mathsf{SymPoly}$ \\
    $h_n$, $h_\lambda$   &        complete homogeneous polynomials $h_n$, and complete homogeneous basis $h_\lambda = h_{\lambda_1} h_{\lambda_2} \cdots h_{\lambda_m}$ in $\mathsf{SymPoly}$ \\
    $p_n$, $p_\lambda$   &        power-sum symmetric polynomials $p_n = \sum_k x_k^n$, and $p_\lambda = \prod_i p_{\lambda_i}$.\\
    $\phi_n(q)$, $\phi(q)$          &        $\phi_n(q) = (q; q)_n = \prod_{i = 1}^n (1 - q^i)$, $\phi(q) = \phi_\infty(q)$ \\
    $\langle V, W\rangle$           &        the Hall inner product of two $S_n$ representations $\langle V, W\rangle = \dim (V \otimes W)^{S_n}$  \\
    $\B_n$, $\P_\lambda$      &       the complete flag variety, and the partial flag variety associated with a partition $\lambda$ \\
    $P_u(M)$      &          the Poincare polynomial $\sum_i \dim H^i(M) (-u)^i$ of a space $M$.
  \end{longtable}
\end{notation}

\subsection{Graded characters and cohomology of flag varieties} % (fold)
\label{sub:graded_characters_and_cohomology_of_flag_varieties}
  \begin{definition}[Frobenius character]
  \label{def: graded character}
  The \emph{Frobenius character} map is a ring isomorphism 
  \[
    \ch\colon \bigoplus_{n = 0}^\infty \Rep_\Q(S_n) \to \mathsf{SymPoly}
  \]
  from the ring of representations of $S_n$ under the induction product, to the ring of symmetric polynomials over $\Q$. 

  Under $\ch$, the Specht module $V_\lambda$ corresponding to a partition $\lambda \vdash n$ is sent to the Schur function $s_\lambda$. Using the $p$-basis of $\mathsf{SymPoly}$, the Frobenius character of an $S_n$ representation $V$ has the following expression
  \[
    \ch(V) = \frac{1}{n!} \sum_{\sigma \in S_n} \tr(\sigma) p_\sigma,
  \]
  where $p_\sigma$ is the power sum symmetric function associated with the cycle type of $\sigma$.

  More generally, if $V = \bigoplus_{i = 0}^\infty V_i$ is a graded $S_n$-representation over $\Q$, we define its graded character $\ch_u$ as an alternating series
  \[
    \ch_u(V) = \sum_{i = 0}^\infty \ch(V_i) (-u)^i,  
  \] 
  where $u$ is a formal variable.
  \end{definition}

  \begin{definition}[Hall inner product]
    There exists an inner product $\langle -, -\rangle$ on the ring of symmetric polynomials $\mathsf{SymPoly}$ where Schur functions $s_\lambda$ serves as an orthonormal basis (while $p_\lambda$ are orthogonal basis). This inner product is compatible with the inner product $\langle U, V\rangle_{S_n}$ of $S_n$-representations under the Frobenius character map. We can extend $\langle -, -\rangle$ bilinearly to graded characters
    \[
      \langle -, -\rangle \colon \mathsf{SymPoly}[[u]] \times \mathsf{SymPoly}[[u]] \to \Q[[u]],
    \]
    where the formal variable $u$ is treated as a scalar for the inner product. 
  \end{definition}

  {\bf Convention:}
  For simplicity, when $S_n$ acts on the cohomology groups of a space $Y$, we write $\ch_u(Y)$ for $\ch_u(H^*(Y))$.

  The cohomology ring of $\GL_n/\T_n$ is isomorphic to the cohomology ring of the flag variety $\B_n$ since $\GL_n/\T_n \to \B_n$ is an affine space bundle. The free $S_n$ action on $\GL_n/\T_n$ induces an $S_n$ action on its cohomology ring. The following presentation is due to Borel \cite{borel1953cohomologie}:
  \[
    H^*(\GL_n/\T_n) =
    H^*(\B_n) = \Q[x_1, x_2, \dots, x_n]/(e_1, e_2, \dots, e_n), \quad \text{$S_n$ permutes $x_i$'s}.
  \]

  % In Garsia--Procesi \cite{garsia1992certain}, the graded $S_n$ characters for the cohomology of Springer fibers were computed. In particular we obtain a convenient formula for the graded $S_n$ characters on a flag variety.
  Here we recall a convenient formula for the $S_n$ characters on the cohomology of a flag variety. (This is a special case of Garsia--Procesi's formula for the characters on the cohomology of Springer fibers, see \cite{garsia1992certain}.)
  \begin{theorem}[Graded character of the flag variety]
  \label{thm: graded character of the flag variety}
    Let $\phi_n(q) = \prod_{i = 1}^n (1 - q^i)$, let $\sp_q(f) = f(1, q, q^2, \dots)$ denote the principal specialization of a symmetric polynomial $f$. We have
    \[
      \ch_u(\B_n) = \sum_{\lambda \vdash n} \phi_n(u^2) \sp_{u^2}(s_\lambda) s_\lambda.
    \]
  \end{theorem}

  \begin{remark}
    In the above theorem, setting $u = 1$ recovers the decomposition of the regular representation $\Q[S_n]$ into Specht modules.
  \end{remark}

  \begin{corollary}
  \label{lem: graded inner product with flag}
    Let $V$ be a graded $S_n$-rep, we have
    \[
      \langle \ch_u(\B_n), V \rangle = \phi_n(u^2) \sp_{u^2}(\ch_u(V)).
    \]
    Consequently, for a topological space $X$, we have
    \begin{equation}
    \label{poincare polynomial of SnX}
      P_u(S_n(X)) = \phi_n(u^2) \sp_{u^2}(\ch_u(X^n)).
    \end{equation}
  \end{corollary}

  % \begin{corollary}
  % \label{cor: point counts on fibers of p1}
  %   Let $S_\lambda$ be the Young subgroup of $S_n$ corresponding to a partition $\lambda \vdash n$, the poincare series of the quotient $(\GL_n/\T_n)/S_\lambda$ equals
  %   \[
  %     \langle \ch_u(\B_n), h_\lambda\rangle = \phi_n(u^2) \sp_{u^2}(h_\lambda)
  %   \]
  % \end{corollary}

  To compute the number of $\Fq$-points on $S_n(X)$, we need to consider the $S_n$ character on $H^*(X^n)$ along with an endomorphism $F$ (the arithmetic Frobenius over $\Fq$). This motivates the following enhancement of \Cref{def: graded character}.

  \begin{definition}[$F$-enhanced Frobenius character]
  \label{def: F-enhanced character}
    Let $V$ be a graded $S_n$ representation $V = \bigoplus_{i \geq 0} V_i$ over $\Q$, along with a graded endomorphism $F$ which commutes with the action of $S_n$. We call such $V$ a \emph{graded $S_n$-rep enhanced by $F$.} 
    The \emph{$F$-enhanced Frobenius character} of $V$ to is defined as
    \[
      \ch_{F,u}(V) = \sum_{i \geq 0} (-u)^i \sum_{\sigma \in S_n} \frac{1}{n!} \tr(F\sigma | V_i) p_\sigma,
    \]
    where $p_\sigma$ is the power sum symmetric function associated with the cycle type of $\sigma$. Note that $\tr(F\sigma)$ only depends on the cycle type of $\sigma$ due to the commutativity of $F$ and $S_n$, thus the inner sum is essentially over the conjugacy classes of $\sigma$. 
  \end{definition}

  \begin{remark}
    When $F = \mathrm{id}$, $\ch_{F, u}(V) = \ch_u(V)$ reduces to the graded character defined in \Cref{def: graded character}.
  \end{remark}

  \begin{remark}
    Let $F = F_s + F_n$ be the Jordan decomposition of $F$ on $V$, where $F_s$ is the semi-simple part. We clearly have $\tr(F\sigma | V_i) = \tr(F_s\sigma | V_i)$. Hence for practical purposes we may always assume $F$ is semi-simple when computing $\ch_{F, u}(V)$.
  \end{remark}

  \begin{lemma}
  \label{lem: F-enhanced inner product}
    Let $V$ a graded $S_n$-rep enhanced by $F_1$, and $W$ a graded $S_n$-rep enhanced by $F_2$. The tensor product $V \otimes W$ is naturally bi-graded, and enhanced by an endomorphism $F = F_1 \otimes F_2$. Thus $F$ naturally acts on the bi-graded vector space $(V \otimes W)^{S_n}$ as well. We have
    \[
      \sum_{i, j \geq 0}(-u_1)^i (-u_2)^j \tr(F | (V_i \otimes W_j)^{S_n}) = \langle \ch_{F_1, u_1}(V), \ch_{F_2, u_2}(W) \rangle.
    \]
    where the inner product on the right hand side is the Hall inner product extended bi-linearly to graded characters.
  \end{lemma}
  \begin{proof}
    As both sides of the equality are additive for short exact sequences in $V$ and $W$ (respectively), it suffices to prove it in the case when $V$ and $W$ are irreducible $S_n$-reps, and both $F_1$ and $F_2$ act as scalars (upon base change to $\overline{\Q}$). In which case the equality can be checked directly using the definition of $\ch_{F, u}$.
  \end{proof}

  Let $X$ be a variety over $\Fq$, the cohomology groups $H^*(X^n)$ form a graded $S_n$ representation with a commuting arithmetic Frobenius endomorphism $F$ commuting with $S_n$. Same for the space $\GL_n/\T_n$. The number of $\Fq$ points on the quotient $S_n(X)$ is determined by the alternating trace of $F$ on its cohomologies, which can be computed using the previous lemma.
  \begin{corollary}
  \label{cor: point counts on X^n x_Sn GLn/Tn}
    The number of $\Fq$ points on $S_n(X)$ equals
    \[
      |\GL_n(\Fq)| \cdot \sp_{q^{-1}}(\ch_{F, 1}(X^n)).
    \]
  \end{corollary}
  \begin{proof}
    By \Cref{lem: F-enhanced inner product} and the Behrend trace formula (trace formula for cohomology groups on smooth varieties), we have
    \[
    \begin{aligned}
      S_n(X)(\Fq) & = q^{n^2} \langle \ch_{F, 1}(X^n), \ch_{F, 1}(\GL_n/\T_n)\rangle \\
      & = q^{n^2} \phi_n(q^{-1})\sp_{q^{-1}}(\ch_{F, 1}(X^n)) \\
      & = |\GL_n(\Fq)| \cdot \sp_{q^{-1}}(\ch_{F, 1}(X^n))
    \end{aligned}
    \]
    where the second equality follows from the same calculation as in \Cref{lem: graded inner product with flag}.
  \end{proof}

  % subsection cohomology_of_flag_variety (end)

\subsection{Generating Series and Zeta function} % (fold)
\label{sub: generating series and zeta function}
  While we have not used the (induction product) ring structure on $\mathsf{SymPoly}$, for what follows the ring structure becomes essential. Let $(V,F)$ be a fixed graded vector space with a graded endomorphism $F$. By abuse of notation, we denote the natural endomorphism $F^{\otimes n}$ on $V^{\otimes n}$ still by $F$. We equip the \textbf{signed $S_n$ representation} on $V^{\otimes n}$ as the left $S_n$ action satisfying:
  $$
    (ij)(\cdots v_i \otimes v_j \cdots) = (-1)^{\deg(v_i)\deg(v_j)}(\cdots v_j \otimes v_i \cdots)
  $$
  for $(ij) \in S_n$ a transposition, and $v_i, v_j$ homogeneous.
  By \cite{macdonald1962symmetric}, such an action exists uniquely. (There is a direct albeit more technical definition in \cite{macdonald1962symmetric}, which shows that it is clearly well-defined, but the current form is enough to give an algorithm to compute the action on any pure tensor.)

  We now compute the generating series of $F$-enhanced characters of $V^{\otimes n}$.

  \begin{lemma}
    Let $(V, F)$ be a graded vector space with a graded endomorphism, let $\mathcal{A} \subset V$ be a homogeneous eigen-basis of the semi-simple part of $F$ on $V$, and $\chi \colon \mathcal{A} \to \overline{\Q}$ be the corresponding eigen-values. 

    Let $\sigma \in S_n$ be a permutation whose cycle type corresponds to a partition (in exponential form) $\lambda = (1^{a_1} 2^{a_2} \cdots n^{a_n})$. Let $\tr_u(F\sigma | V^{\otimes n}) = \sum_i \tr(F\sigma|(V^{\otimes n})_i) (-u)^i$ be the graded trace of $F\sigma$ on $V^{\otimes n}$. We have
    \[
      \tr_u(F\sigma|V^{\otimes n}) = \prod_{i\geq 1} \left(\sum_{\alpha\in \mathcal{A}}  (-1)^{\deg(\alpha)}u^{i \deg(\alpha)}\chi(\alpha)^{i}\right)^{a_i}.
    \]
  \end{lemma}
  \begin{proof}
    As mentioned earlier, we may assume $F = F_s$ is already semi-simple. Thus $V^{\otimes n}$ has a homogeneous basis $\{ \bigotimes_{i = 1}^n \alpha_i \ |\ \text{for }\alpha_i \in \mathcal{A}\}$. Note that under this basis, the matrix for $F\sigma$ is \emph{monomial}, meaning that there's only one non-zero entry for each row and each column. 
    Thus the only contribution to trace comes from basis vectors $\bigotimes_{i = 1}^n \alpha_i$ satisfying
    \[
      \alpha_i = \alpha_{\sigma(i)} \text{ for all }i.
    \]
    In other words, if $\sigma$ is decomposed into $r$ cycles of lengths $\lambda_1, \lambda_2, \dots, \lambda_r$, the set $\{\alpha_i\}_{i = 1}^n$ will only have $r$ distinct elements $\{\beta_j\}_{j = 1}^r$, and we may compute
    \[
    \begin{split}
      F\sigma \left( \bigotimes\nolimits_i \alpha_i \right) & = (-1)^{\sum_{j = 1}^r (\lambda_j - 1) \deg (\beta_j)} F\left( \bigotimes\nolimits_i \alpha_i \right)\\
      & = \prod_{j = 1}^r (-1)^{ (\lambda_j - 1) \deg (\beta_j)} \chi(\beta_j)^{\lambda_j}\left( \bigotimes\nolimits_i \alpha_i \right)
    \end{split}
    \]
    and we get
    \[
    \begin{split}
      \tr_u(F\sigma | V^{\otimes n}) & = \sum_{\beta_1,\dots,\beta_r \in \mathcal{A}} \prod_{j = 1}^r (-1)^{ (\lambda_j - 1) \deg (\beta_j)} \chi(\beta_j)^{\lambda_j}(-u)^{\lambda_j \deg(\beta_j)} \\
      & = \prod_{j = 1}^r \sum_{\beta \in \mathcal{A}}(-1)^{\deg(\beta)} \chi(\beta)^{\lambda_j} u ^{\lambda_j \deg(\beta)} \\
      & = \prod_i \left(\sum_{\beta \in \mathcal{A}}(-1)^{\deg(\beta)} \chi(\beta)^{i} u ^{i \deg(\beta)}\right)^{a_i}.\qedhere
    \end{split}
    \]
  \end{proof}

  % \begin{definition}[Power rule]
  % \label{def: power rule on sympoly}
  %   Let $E = \sum_{i, j} E_{ij} u^i t^j \in \mathsf{SymPoly}[[u, t]]$ be a power series with coefficients being symmetric polynomials. Let $f = \sum_{i \geq 0} f_i u^i \in \Z[ [u]]$ be a power series in $u$ with integer coefficients. We define the exponential $E^f$ using the following rules
  %   \begin{enumerate}
  %     \item for an integer $n$, $E^n$ is the usual $n$-th power.
  %     \item for $E = \sum_{i, j} E_{ij} u^i t^j$, define $E^u$ to be $\sum_{i, j} E_{ij} u^i (ut)^j$.
  %     \item compatibility: $E^{f + g} = E^f E^g$, and $E^{fg} = (E^f)^g$.
  %     \item continuity: for $f = \sum_{i \geq 0} f_i u^i$, $E^f = \prod_{i \geq 0} E^{f_i u^i}$.
  %   \end{enumerate}
  % \end{definition}

  % Denote the generating series $\sum_{n\geq 0} \ch_{F,u}(V^{\otimes n}) t^n$ by $S(V)$, let $0 \to U \to V \to W \to 0$ be an exact sequence of graded vector spaces with graded endomorphism $F$, the next theorem says we expect a simple relation between the power series $S(V) = S(U) \cdot S(W)$. 

  \begin{theorem}
  \label{thm: F enhanced generating series of tensor power}
    We have (by passing to $\overline{\Q}$)
    \begin{equation}
        \sum_{n\geq 0} \ch_{F,u}(V^{\otimes n}) t^n = \prod_{\chi} \left(1 + \sum_{n = 1}^\infty h_n \chi^n t^n \right)^{[V_\chi]},
    \end{equation}
    where $V_\chi = \bigoplus_{i \geq 0} V_{i, \chi}$ is the graded generalized eigenspace associated with the eigenvalue $\chi$ of $F$, and $[V_\chi] = \sum_{i \geq 0} (- u)^i \dim V_{i, \chi}$. The exponentiation on the right hand side simply means:
    \[
      \left(1 + \sum_{n = 1}^\infty h_n \chi^n t^n \right)^{[V_\chi]} = \prod_i \left(1 + \sum_{n = 1}^\infty h_n \chi^n (u^it)^n \right)^{(-1)^i \dim V_{i, \chi}}.
    \]
  \end{theorem}
  \begin{proof}
    Keeping notations the same as in the previous lemma, let $\sigma_\lambda$ denote any permutation of cycle type $\lambda$, and recalling that $z_\lambda=\prod_{i\geq 1}i^{a_i}a_i!$ for $\lambda=1^{a_1}2^{a_2}\dots$ and
    \[ \exp\left(\sum_{n\geq 1}\frac{p_n}{n}t^n\right)=1+\sum_{n\geq 1}h_n t^n,\]
    we have 
    \begin{align*}
        \sum_{n\geq 0} \ch_{F,u}(V^{\otimes n}) t^n&=\sum_{\lambda} \tr_u(F\sigma_\lambda|V^{\otimes \abs{\lambda}})\frac{p_\lambda}{z_\lambda} t^{|\lambda|}\\
        &=\sum_{a_1,a_2,\dots\geq 0}  \prod_{i\geq 1} \left(\sum_{\alpha\in \mathcal{A}}  (-1)^{\deg(\alpha)}u^{i \deg(\alpha)}\chi(\alpha)^{i} t^i \right)^{a_i} \frac{p_i^{a_i}}{i^{a_i}a_i!}
        \\
        &=\prod_{i\geq 1} \sum_{a\geq 0}\left(\frac{p_i}{i}\sum_{\alpha\in \mathcal{A}}  (-1)^{\deg(\alpha)}u^{i \deg(\alpha)}\chi(\alpha)^i t^i \right)^a/a!\\
        &=\exp\left(\sum_{i\geq 1}\frac{p_i}{i}\sum_{\alpha\in \mathcal{A}}  (-1)^{\deg(\alpha)}u^{i \deg(\alpha)}\chi(\alpha)^i t^i\right)\\
        &=\exp\left(\sum_{\alpha\in \mathcal{A}}(-1)^{\deg(\alpha)}\sum_{i\geq 1}\frac{p_i (u^{\deg(\alpha)}\chi(\alpha))^i t^i}{i}\right)\\
        &=\prod_{\alpha\in \mathcal{A}} \left(1+\sum_{n\geq 1} h_n (u^{\deg(\alpha)}\chi(\alpha))^n t^n\right)^{(-1)^{\deg(\alpha)}},
    \end{align*}
    which is equivalent to what we want.
  \end{proof}

  % \YW{show that we arrived at new proofs of Yifeng's old results}
  \begin{definition}[The groupoid generating series $Z_X(t)$]
  \label{def: groupoid generating series Z_X(t)}
    Let $X$ be a quasi-projective variety over a finite field of size $q$. Recall that the general linear group $\GL_n(\Fq)$ acts on the set of rational points $C_n(X)(\Fq)$. We define $Z_X(t)$ to be the {\it groupoid generating series}
    \[
      Z_X(t) = 1 + \sum_{n = 1}^\infty \frac{|C_n(X)(\Fq)|}{|\GL_n(\Fq)|} t^n.
    \]
  \end{definition}

  We arrive at a new combinatorial proof of the following result.
  \begin{theorem}[\cite{huang2023mutually}]
  \label{thm: zeta product formula}
    When $X$ is a smooth curve over a finite field of size $q$, the groupoid generating series $Z_X(t)$ has an infinite product formula
    \[
      Z_X(t) = \prod_{i = 1}^\infty \zeta_X(q^{-i}t).
    \]
  \end{theorem}
  \begin{proof}
    Note that by our main result \Cref{thm: main result} $C_n(X)$ and $S_n(X)$ have the same cohomology. In the case when $X$ is a smooth curve, $C_n(X)$ is a smooth variety of dimension $n^2$, thus by trace formula we have $C_n(X)(\Fq) = S_n(X)(\Fq)$. Now we can evaluate $Z_X(t)$ using \Cref{cor: point counts on X^n x_Sn GLn/Tn}
    \[
    \begin{aligned}
      1 + \sum_{n = 1}^\infty \frac{S_n(X)(\Fq)}{\GL_n(\Fq)} t^n & = 1 + \sum_{n = 1}^\infty \sp_{q^{-1}}(\ch_{F, 1} (X^n)) t^n \\
      & = \sp_{q^{-1}}\left( \sum_{n \geq 0 } \ch_{F, 1}(X^n) t^n\right).
    \end{aligned}
    \]
    % Let $V = H^1(X)$, $\varepsilon = 0$ or $1$ depending on whether $X$ is affine or projective. Let $P(t) = \det(1 - tF) = \prod_\chi (1 - \chi t)$ be the characteristic polynomial of $F$ acting on $V$. 
    By \Cref{thm: F enhanced generating series of tensor power}, and the fact that $\sp_{q^{-1}}$ is a ring homomorphism, we can further evaluate (where $V = \bigoplus_r H^r(X)$)
    \[
    \begin{aligned}
      \sp_{q^{-1}}\left( \sum_{n \geq 0 } \ch_{F, 1}(X^n) t^n\right) 
      & = \prod_{\chi} \sp_{q^{-1}}\left(1 + \sum_{n = 1}^\infty h_n \chi^n t^n \right)^{[V_\chi]} \\
      & = \prod_{\chi} \prod_{i \geq 0} \left(\frac{1}{1 - q^{-i} \chi t}\right)^{[V_\chi]} \\
      & = \prod_{i \geq 1} \zeta_X(q^{-i}t),
    \end{aligned}
    \]
    where the second equality follows from the sum-product formula for generating series of $h_n$ ($x_n$ are the polynomial variables in $\mathsf{SymPoly}$)
    \[
      \sum_{n \geq 0} h_n t^n = \prod_{n = 1}^\infty \frac{1}{1 - x_n t}
    \]
    and the third equality follows from the product formula of the Weil zeta function
    \[
      \zeta_X(t) = \prod_\chi \left(\frac{1}{1 - \chi qt}\right)^{[V_\chi]}. \qedhere
    \]
  \end{proof}

\subsection{Betti Zeta function} % (fold)
\label{sub: betti zeta function}
  Finally we prove a generating series for Poincare series of $C_n(X)$ using similar computations as above. To fully demonstrate the parallelism, let us also define the Betti analog of the zeta function.
  \begin{definition}[Betti analog of zeta function]
    Let $X$ be a variety over an algebraically closed field, we define its Betti zeta function $\zeta_X^\mathrm{B}(t)$ to be the generating series of Poincare series of $\Sym^n X$
    \[
      \zeta_X^\mathrm{B}(t) = \sum_{n = 0}^\infty P_u(\Sym^n X)t^n.
    \]
  \end{definition}

  The following Betti analog of Weil conjecture was proved by I. G. Macdonald in 1962.
  \begin{theorem}[\cite{macdonald1962poincare}]
  \label{thm: Betti analog weil conjecture Macdonald formula}
    Let $X$ be a variety over an algebraically closed field, let $P_u(X) = \sum_{i \geq 0} \dim H^i(X) (-u)^i$ be the Poincare polynomial of $X$. We have 
    \[
      \zeta_X^\mathrm{B}(t) = \left(\frac{1}{1 - t}\right)^{P_u(X)} = \prod_{i \geq 0} \left(\frac{1}{1 - u^i t}\right)^{(-1)^i \dim H^i(X)}
    \]
  \end{theorem}

  Parallel to \Cref{thm: zeta product formula}, we have the following Betti version.
  \begin{theorem}
  \label{thm: generating series for Betti numbers}
  We have the following product formula 
  \begin{equation}
  \label{eqn: generating series Betti numbers}
    1 + \sum_{n = 1}^\infty P_u(\Coh_n(X))t^n = \prod_{i, j \geq 0} \left(\frac{1}{1 - u^{2i + j} t}\right)^{(-1)^j \dim H^j(X)} = \prod_{i = 0}^\infty \zeta_X^\mathrm{B}(u^{2i}t).
  \end{equation}
  \end{theorem}
  \begin{proof}
    The Poincare polynomial of $C_n(X)$ equals that of $S_n(X)$, thus by \Cref{poincare polynomial of SnX} we have$P_u(C_n(X)) = \phi_n(u^2)\sp_{u^2}(\ch_u(X^n))$. Furthermore, the equivariant formality \Cref{equivariant formality} implies that $P_u(\Coh_n(X)) = P_u(C_n(X)) P_u(B\GL_n)$. Thus $P_u(\Coh_n(X))$ simply equals to $\sp_{u^2}(\ch_u(X^n))$.

    Now an application of \Cref{thm: F enhanced generating series of tensor power} (for $F = id$), and the same computation as in \Cref{thm: zeta product formula} shows
    \[
    1 + \sum_{n = 1}^\infty \sp_{u^2}(\ch_u (X^n)) t^n = \prod_{i, j \geq 0} \left(\frac{1}{1 - u^{2i + j} t}\right)^{(-1)^j \dim H^j(X)} = \prod_{i = 0}^\infty \zeta_X^\mathrm{B}(u^{2i}t).
    \]
    where the last equality follows from Macdonald's formula (\Cref{thm: Betti analog weil conjecture Macdonald formula}).
  \end{proof}

  \begin{corollary}
  \label{cor: Betti number stabilizes}
    When $X$ is connected, the Betti numbers of $C_n(X)$ stabilizes as $n \to \infty$
    \[
      \lim_{n \to \infty} P_u\left(C_n(X)\right) = \phi(u^2) \lim_{t \to 1} (t - 1) \prod_{i = 0}^\infty \zeta_X^\mathrm{B}(u^{2i}t).
    \]
  \end{corollary}
  \begin{proof}
    The stability of $P_u(C_n(X))/\phi_n(u^2)$ follows from representation stability of $H^*(X^n)$ and $H^*(\GL_n /\T_n)$. The residue at $t = 1$ of \Cref{eqn: generating series Betti numbers} creates a telescoping series, hence the corollary.
  \end{proof}
% subsection poincare_polynomials (end)

%!TEX root=./short_version.tex
\section{A Hermitian variant over $\R$}
\label{sec: a hermitian variant}
We conclude by presenting a Hermitian version of our main result. In this section we assume our ground field $k = \R$. 

% \YW{Edit the following, use more precise language. Remove mentioning of Boson-Fermion duality.}
Let $X = \Spec R$ be an affine scheme of finite type defined over $k$. The algebraic set $C_n(X)(\C)$ is now equipped with an involution $(-)^\dagger$, sending commuting complex matrices to their conjugate transposes. 

\begin{definition}
The topological space of Hermitian points $HC_n(X) \subset C_n(X)(\C)$ is the fixed locus of the involution $(-)^\dagger$.
\end{definition}

More canonically, $(-)^\dagger$ takes a homomorphism $\phi \colon R \to \Mn(\C)$ to the composite $R \stackrel{\phi}{\to} \Mn(\C) \stackrel{{}^\dagger}{\to} \Mn(\C)$. 
We remark that one may generalize the definition of $HC_n(X)$ to quasi-projective schemes $X$ over $\R$ as well, and results in this section will hold in that generality. But for the sake of simplicity we continue to assume $X$ is affine.

In down to earth terms, $HC_n(X)$ is the set of commuting Hermitian matrices satisfying the defining equations of $X$. Since the spectrum of Hermitian matrices consists of real numbers, the support map now restricts to $HC_n(X) \to \Sym^n (X(\R))$.

We define the Hermitian semi-simple counterpart here using the unitary group $\U_n$ and its subgroup of diagonal unitary matrices $\UT_n$.

\begin{definition}
The Hermitian semi-simple counterpart $HS_n(X)$ is the quotient of $X(\R)^n \times \U_n/\UT_n$ by the diagonal $S_n$ action.
\end{definition}

Analogous to the map $\sigma$, we have a morphism $\sigma_H \colon HS_n(X) \to HC_n(X)$ over $\Sym^n (X(\R))$. This map takes an ordered tuple of real points and a decomposition of $\C^n$ into orthogonal lines (parametrized by $\U_n/\UT_n$) and constructs the corresponding semisimple Hermitian operator. 
% Importantly, the map $\sigma_H$ is proper, as it has a proper source, thus we may use proper base change theorem to analyze the stalks of the direct image of constant sheaves along $HS_n(X) \to \Sym^n X(\R)$ and $HC_n(X) \to \Sym^n X(\R)$.

\begin{remark}
  The spaces $HS_n(X)$ and $HC_n(X)$ are in fact real Galois twists of the original $S_n(X)$ and $C_n(X)$. On $S_n(X)$ we twist the complex conjugation by the inverse transpose operation on the factor $\GL_n/\T_n$, and on $C_n(X)$ we twist by matrix transpose. Note that the morphism $\sigma$ is invariant under these involutions, therefore $\sigma_H \colon HS_n(X) \to HC_n(X)$ is merely the Galois twist of $\sigma$. 
\end{remark}

\begin{theorem}[\Cref{thm: hermitian version}]
  Let $p_1 \colon HS_n(X) \to \Sym^n X(\R)$ and $p_2 \colon HC_n(X) \to \Sym^n X(\R)$ denote the canonical projections, we have a sheaf theoretic isomorphism $\sigma_H^* \colon Rp_{2*}\uQ \iso Rp_{1*}\uQ$. In particular, $\sigma_H$ induces isomorphism between the rational cohomology of $HC_n(X)$ and $HS_n(X)$.
\end{theorem}

\begin{proof}
We analyze the map $\sigma_H$ fiberwise over $\Sym^n (X(\R))$. Since commuting Hermitian matrices are simultaneously diagonalizable by a unitary transformation, we have:
\begin{enumerate}
    \item The map $\sigma_H$ is surjective.
    \item All finite sheaves appearing in $HC_n(X)$ are semisimple (they have no nilpotent parts).
    \item The following maps in the diagram are all proper maps
    \[
    \begin{tikzcd}
      HS_n(X) \ar[rr, "\sigma_H"] \ar[rd, "p_1"'] & & HC_n(X) \ar[ld, "p_2"] \\
      & \Sym^n X(\R) &
    \end{tikzcd}
    \]
\end{enumerate}

Let $z_\lambda \in \Sym^n (X(\R))$ correspond to a partition $\lambda \vdash n$. The fiber of $HC_n(X) \to \Sym^n (X(\R))$ above $z_\lambda$ consists of matrices unitarily equivalent to a fixed diagonal matrix with eigenvalues determined by $z_\lambda$. This fiber is a single orbit under the unitary group $\U_n$. The stabilizer is the product of unitary groups corresponding to the eigenspaces, $\prod_i \U_{\lambda_i}$. Thus, the fiber is isomorphic to the homogeneous space $\U_n/\prod_i \U_{\lambda_i}$.

The fiber of $HS_n(X) \to \Sym^n (X(\R))$ above $z_\lambda$ is isomorphic to $(\U_n/\UT_n)/S_\lambda$.

The map $\sigma_H$ restricts to a surjective map on the fibers:
\[
\sigma_H \colon (\U_n/\UT_n)/S_\lambda \to \U_n/\prod_i \U_{\lambda_i}.
\]
This is a fibration. The fibers are isomorphic to $\left(\prod_i \U_{\lambda_i}\right) / N(\prod_i \UT_{\lambda_i})$. Similar to the complex case, the spaces $\U_m/N(\UT_m)$ have trivial rational cohomology. Therefore, $\sigma_H$ induces an isomorphism on the cohomology of the fibers. By proper base change theorem, fiberwise cohomology isomorphism implies isomorphism of stalks $(Rp_{2*}\uQ)_{z_\lambda} \to (Rp_{1*}\uQ)_{z_\lambda}$, thus we obtain the sheaf theoretic isomorphism.
\end{proof}

\printbibliography
\end{document}